\documentclass[preprint,11pt]{elsarticle}

\usepackage{amssymb}
\usepackage{amsthm}
\usepackage{amsfonts}
\usepackage{amsmath}

\usepackage{epsfig}
\usepackage{graphicx}
\usepackage{color}

\renewcommand{\Re}{\mathbb{R}}

\renewcommand{\d}{\displaystyle}

\newcommand{\BB}{{\sf B}}
\newcommand{\CC}{{\sf C}}

\newcommand{\FF}{{\sf F}}
\newcommand{\HH}{{\sf H}}
\newcommand{\GG}{{\sf G}}
\newcommand{\II}{{\sf I}}

\newcommand{\LL}{{\sf L}}
\newcommand{\MM}{{\sf M}}

\newcommand{\PiR}{{\sf \Pi_R}}
\newcommand{\GR}{{\sf G_R}}
\newtheorem{lemma}{Lemma}
\newtheorem{prop}{Proposition}
\newtheorem{theorem}{Theorem}
\newtheorem{corollary}{Corollary}

\def\s{{\sf s}}
\def\LL{{\sf L}}

\renewcommand{\d}{\displaystyle}

\begin{document}

\begin{frontmatter}



\title{The Group Inverse of extended Symmetric and
Periodic Jacobi Matrices}


\author{S. Gago}

\address{}

\begin{abstract}
In this work, we explicitly compute the group inverse of symmetric and periodic Jacobi matrices with constant elements that have been extended by adding a row and a column conveniently defined. For this purpose, we interpret such matrices as the combinatorial Laplacian of a non--complete wheel that has been obtained by adding  a vertex to a cycle and some edges conveniently chosen. The obtained group inverse is an incomplete block  matrix with a block Toeplitz structure. In addition, we obtain the effective resistances and the Kirchhoff index of non--complete wheels. 
\end{abstract}

\end{frontmatter}


\section{Introduction and notation}
\label{}

The invertibility of nonsingular tridiagonal or block tridiagonal matrices has been  studied in recent years; see for instance \cite{BG03, BHR14, LHLL10, M92}. Moreover, explicit inverses are known  in some cases, for instance when the tridiagonal matrix is symmetric with constant diagonals and subject to some restrictions. In \cite{FP01}, da Fonseca and Petronilho obtained explicit inverses of tridiagonal $2$--Toeplitz and $3$--Toeplitz matrices which generalize
some well-known results concerning the inverse of a tridiagonal Toeplitz matrix. The techniques used in the mentioned results are mainly based on  the theory of orthogonal polynomials.   For the singular case, there is also a big amount of work, for instance in \cite{BCEM10} the authors carried out an exhaustive analysis of the genera\-lized inverses of singular irreducible symmetric $M$--matrices. The key idea of the approach was to identify any symmetric $M$--matrix with a positive semi--definite Schr\"odinger operator on a connected network whose conductances are given by the off--diagonal elements of the $M$--matrix. Moreover, explicit expressions for the group inverses in the cases of tridiagonal matrices and some circulant matrices were obtained in \cite{CEGJM15}.

In this work, we present a new formula  for the group inverse of an extended symmetric and periodic Jacobi matrix with constant elements. By extended we mean that a row and a column have been added to the matrix. The idea is to see the extended symmetric and periodic Jacobi matrix as the combinatorial Laplacian of a non--complete wheel and hence to obtain the group inverse, by using some previous results obtained in \cite{CEGM16}.  A non--complete wheel is a wheel where the central vertex is connected to a few vertices of the cycle. This kind of networks has many applications in Computer Science as the central vertex is called a hub (see \cite{BJR13}). Therefore, the non--complete wheel can be seen as a cycle with an added vertex and some new edges and hence, the result of \cite{CEGM16} can be applied. Actually, a particular example of this network can be found in \cite{CEGM16}. Finally, we use the formula for the group inverse to give the effective resistances and the Kirchhoff Index of a non--complete wheel.

In the following, the triple $\Gamma=(V,E,c)$ denotes a finite network, that is, a finite graph without loops nor multiple edges, with vertex set $V=\{1,\dots,n\}$ and edge set $E$, where each edge $e_{ij}=\{i,j\}$ has associated a conductance $c_{ij}>0$. The standard inner product on $\Re^n$ is denoted by $\langle \cdot, \cdot \rangle$, thus, if ${\sf u}$, ${\sf v}\in \Re^n$, then $\langle {\sf u}, {\sf v}\rangle=\sum_{k=1}^n u_k v_k $. For any $i=1,\dots,n$ we denote by ${\sf e}_i$ the $i$--th vector of the standard basis of $\Re^n$, by ${\sf j}_n$ the all--ones vector of dimension $n$ and by ${\sf J}_{n,m}$ the all--ones matrix of size $(n,m)$. Moreover $T_n(q)$, $U_n(q)$ and $V_n(q)$ denote respectively the $n$-th Chebyshev polynomials of first, second and third kind, that is, the Chebyshev polynomials satisfying $T_0(q)=U_0(q)=V_0(q)=1$, $T_1(q)=q$, $U_1(q)=2q$ and $V_1(q)=2q-1$, for any $n\in \mathbb{N}$.

A matrix is called a {\it block Toeplitz matrix} iff it is  a block matrix, which contains blocks that are repeated in each descending diagonal from left to right, as a Toeplitz matrix in which each descending diagonal from left to right is constant. A matrix ${\sf A}$ of order $n=md+p$, $0<p<d$, is called {\it incomplete block matrix} if  it is partitioned from the top--left--hand corner using
$d\times d$ submatrices as far as possible. Thus
$$
{\sf A}=\left[\begin{matrix}{\sf A}_{11}&\cdots &{\sf A}_{1m} & {\sf A}_{1m+1}\\[2ex]
\vdots&\ddots &\vdots & \vdots\\[2ex]
{\sf A}_{m1}&\cdots &{\sf A}_{mm} & {\sf A}_{mm+1}\\[2ex]
{\sf A}_{m+11}&\cdots &{\sf A}_{m+1m} & {\sf A}_{m+1m+1}\end{matrix}\right],$$
where ${\sf A}_{ij}$, $i,j=1,2,\ldots,m$ is an $d\times d$ matrix, ${\sf A}_{im+1}$ and ${\sf A}^\top_{m+1 j}$, $i,j=1,2,\ldots,m$ are $d\times p$ matrices and ${\sf A}_{m+1m+1}$ is an $p\times p$ matrix.

The combinatorial Laplacian or simply the Laplacian of the network $\Gamma$ is the matrix ${\sf L}=({\sf L}_{ij})$, where ${\sf L}_{ii}=\sum_{k=1}^n c_{ik}$, and ${\sf L}_{ij}=-c_{ij}$ when $i\neq j$. It is well--known that the Laplacian is singular, symmetric and positive semi--definite. Moreover ${\sf L}{\sf u}=0$ iff ${\sf u}$ is proportional to vector ${\sf j}_n$. The group inverse of the Laplacian, $\LL^{\#}$, is known as the Green matrix of the network $\Gamma$, and from now on will be denoted by ${\sf G}$.

For every pair of vertices $\{i,j\}$ we define the dipole between them as the vector ${\sf \pi}_{ij}={\sf e}_i-{\sf e}_j$. Observe that ${\sf \pi}_{ii}=0$ and  ${\sf \pi}_{ij}=-{\sf \pi}_{ji}$.
The \textit{effective resistance}  between two vertices $\{i,j\}$ of a network $\Gamma$ can be computed by using the following formula of \cite{BCEG09}:
\begin{equation}
\label{resistance}
R(i,j)=\langle {\sf G} {\sf \pi}_{ij},  {\sf \pi}_{ij}\rangle= ({\sf G})_{ii}+({\sf G})_{jj}-2({\sf G})_{ij},
\end{equation}
and the total resistance of the network or \textit{Kirchhoff index} can be also computed as follows:
$$K(\Gamma)=\frac{1}{2}\sum_{i,j\in V} R(i,j) = n \sum_{i=1}^n  ({\sf G})_{ii}.$$

Given $m>1$, $d\geq 1$, we consider a cycle $C_{n}$ on $n=md$ vertices 
with constant conductances $c=c(i,i+1)>0$ for any $i=1,\dots,n-1$ and $c=c(n,1)$. A \textit{non--complete wheel  $(n,m)$-$W$,} is a network obtained from $C_n$ by adding a  new vertex $n+1$ to $m$ of the vertices of the cycle placed at the same distance, $d$, with new conductances $a=c(n+1,1+d(i-1))$, for any $i=1,\ldots,m$.

It is known (see for instance \cite{BCEM10}) that the Laplacian matrix of the cycle  $C_{n}$, is a circulant matrix ${\sf L}=circ(2c,-c,0,\dots,0,-c)$ and its group inverse is
$$\left({\sf G}\right)_{ij}=\frac{1}{12cn}\left(n^2-1-6|i-j|(n-|i-j|)\right),\quad i,j=1,\dots,n.$$

\section{The group inverse of non--complete wheels}

In this section we give an explicit expression for the group inverse of extended symmetric and
periodic Jacobi matrices. Not surprisingly, the expression is an incomplete block Toeplitz matrix whose coefficients involve Chebyshev polinomials.  In order to obtain the claimed expression we consider the Laplacian matrix of the non--complete wheel  network $(n,m)$-$W$  in terms of the Laplacian of the base cycle. Then, we obtain¡ the matrix of order $n+1$ given by
$$
 {\sf L'}=
 \begin {bmatrix} {\sf L}+{\sf D}& {\sf s}\\[1ex]
 {\sf s}^\top & \alpha\\[1ex]
 \end {bmatrix},
 $$
where ${\sf D}$ is a diagonal matrix whose non null elements, $a$, are placed at the ($1+d(i-1)$)--elements of the diagonal, $i=1,\ldots,m$, ${\sf s}=-a\sum\limits_{i=1}^m {\sf e}_{1+d(i-1)}$ and $\alpha = ma$.

\begin{theorem}\label{theo.Main}
The group inverse of $\LL'$ is
$$
 ({\sf L}')^{\#}=
\left(
\begin{array}{cc}
{\sf L}'_{11}
&{\sf L}'_{12}\\[1ex]
 ({\sf L}'_{12})^\top
&{\sf L}'_{22}
\end{array}
\right),
$$
where
$$\begin{array}{rl}
\hspace{-.25cm}\displaystyle {\sf L}'_{22}=&
\hspace{-.25cm}\displaystyle
\frac{12cdn+an(d^2-1)}{12ac (n+1)^2},
 \\[3ex]
\hspace{-.25cm}\displaystyle{\sf L}'_{12}=
&\displaystyle{\sf v}\otimes {\sf j}_m,
 \end{array}
$$
with
$$
({\sf v})_{i}=\frac{a(d+1)[n(d+5)+6]-12cd}{12ac(n+1)^2}+\frac{i(i-d-2)}{2c(n+1)}, \; \textrm{for any}\;\; i=1,\dots,d,
 $$
and where ${\sf L}'_{11}$ is  the block Toeplitz matrix
$$
{\sf L}'_{11}=
\begin {bmatrix}{\sf N}_1& {\sf N}_2  & \dots & {\sf N}_m \\[1ex]
 {\sf N}_m & {\sf N}_1 & \ddots& \vdots\\[1ex]
\vdots & \ddots &\ddots&  {\sf N}_2  \\[1ex]
{\sf N}_2 & \dots & {\sf N}_m  & {\sf N}_1
 \end {bmatrix}$$
such that  any submatrix ${\sf N}_k,$ $k=1,\dots,m$, has entries 
$$\begin{array}{rl}
\hspace{-.25cm}({\sf N}_k)_{i\,j}=
&\hspace{-.25cm}\displaystyle \dfrac{1}{2cn}\Bigg(-\left| d(k-1)+j-i\right|\left(n-\left|  d(k-1)+j-i\right|\right)
\\[3ex]
+&\hspace{-.25cm}\displaystyle\dfrac{1}{T_m(q)-1}
\bigg[ n (j-i)  \left( V_{k-1}(q)  -V_{m-k}(q) \right)
\\[3ex]
\hspace{-.25cm}\displaystyle- &\hspace{-.25cm}\displaystyle
  \Big(\frac{an}{c} ( i-1) ( j-1-d) -nd \Big)  \left( U_{k-2}(q) +U_{m-k}(q) \right)\bigg] 
 \\[3ex]
 \hspace{-.25cm}\displaystyle +&\hspace{-.25cm}\displaystyle
\frac{1}{n+1} \left[ (n(d-2)+2d-j )j+ (d+2-i)i\right]
 \\[3ex]
 \hspace{-.25cm}\displaystyle+ &\hspace{-.25cm}\displaystyle
 (n+2i-2kd)j-(n+2+3d-2kd)i 
 \\[2ex]
  \hspace{-.25cm}\displaystyle +&\hspace{-.25cm}\displaystyle
 d(k-1)\left[n-d(k-1)\right]-\frac{2cd}{a}
  \\[3ex]
 \hspace{-.25cm}\displaystyle +&\hspace{-.25cm}\displaystyle
 \frac{a(d^2-1)+12cd}{6a(n+1)^2}+\frac{n(d+11)(d+1)-d^2+1}{6(n+1)}
 \Bigg),
  \end{array}$$
with $i, j=1,\dots,d$ and  $q=\dfrac{ad}{2c}+1$.
\end{theorem}
\begin{proof}

In order to obtain $({\sf L}')^{\#}$, we use  \cite[Corollary 1]{CEGM16} which reads:
$$
 ({\sf L}')^{\#}=\frac{n^2}{\alpha^2(1+n)^2}
\left(
\begin{array}{cc}
\CC\MM\CC^{\top}+\frac{\alpha}{n^2} {\sf J}_n
& -\frac{\alpha}{n}\,{\sf j}_n-\CC\MM\s\\[1ex]
 -\frac{\alpha}{n}\,{\sf j}_n^{\top}-\s^{\top}\MM\CC^{\top}
&\alpha+\s^{\top}\MM\s
\end{array}
\right),
$$
where
$$\CC= \dfrac{1}{n}\left( \alpha(n+1) \II_n+{\sf j}_n\s^{\top}\right),$$
 and ${\sf M}=\GG-\GG{\sf \Pi}(\II+{\sf \Pi}^\top \GG {\sf \Pi})^{-1} {\sf \Pi}^\top \GG$, where the elements of 
 ${\sf \Pi}\in \mathcal{M}_{n,\frac{m(m+1)}{2}}$ for any $k=1,\dots,m$, $h=k,\dots,m,$ are
 $$
 ({\sf \Pi})_{ij}=
   \sqrt{\frac{a}{m}}
   \cdot
 \left\{
 \begin{array}{rll}
1& \hbox{if}\quad i=1+(k-1)d,& j=(2m-k)(k-1)/2+h,\\[1ex]
 - 1& \hbox{if}\quad i=1+dh,& j=(2m-k)(k-1)/2+h,\\[1ex]
 0 & &\hbox{otherwise.}
 \end{array}
 \right.
 $$
So, to obtain the result we basically need  to compute matrix
 $${\sf M}=\GG-\GG{\sf \Pi}(\II+{\sf \Pi}^\top \GG {\sf \Pi})^{-1} {\sf \Pi}^\top \GG.$$
To do this we need to calculate the different involved matrices and their products in different steps. For the sake of readability, we  develop the proofs in Section 4, throughout some technical lemmas. 
\begin{enumerate}[i) ]
\item In order to obtain $\BB=(\II+{\sf \Pi}^\top \GG{\sf \Pi})^{-1},$ we use a reduced form
$${\sf \Pi}^\top\GG {\sf \Pi}= \PiR^\top \GR \PiR,$$
where ${\sf G_{R}}$ is the $m\times m$ submatrix of ${\sf G}-\frac{n^2-1}{12cn}{\sf J}_n$ whose rows and columns are placed at $1+(k-1)d$, for any $k=1,\dots,m$, and where ${\sf \Pi_R}$ is the submatrix of non-zero elements of ${\sf \Pi}$, that is, ${\sf \Pi_R}\in \mathcal{M}_{m+1,\frac{m(m+1)}{2}}$ and for any $k=1,\dots,m$, $h=k,\dots,m,$ its elements are
 $$
 (\PiR)_{ij}=
   \sqrt{\frac{a}{m}}
   \cdot
  \left\{
 \begin{array}{rll}
1& \hbox{if}\quad i=k,& j=(2m-k)(k-1)/2+h,\\[1ex]
 - 1& \hbox{if}\quad i=1+h,& j=(2m-k)(k-1)/2+h,\\[1ex]
 0 && \hbox{otherwise.}
 \end{array}
 \right.
  $$
\item Now we apply Woodbury's formula, see \cite{Wo50}, to get 
$${\sf B}=(\II+\PiR^\top \GR \PiR)^{-1}=\II-\PiR^\top(\GR^{-1}+\PiR\PiR^\top)^{-1}\PiR.$$
Thus, ${\sf M}$ becomes
\begin{eqnarray*}
{\sf M}&=&\GG-\GG{\sf \Pi}[\II-\PiR^\top(\GR^{-1}+\PiR\PiR^\top)^{-1}\PiR]{\sf \Pi}^\top \GG\\
&=&\GG-\GG[{\sf \Pi}{\sf \Pi}^\top-{\sf \Pi}\PiR^\top(\GR^{-1}+\PiR\PiR^\top)^{-1}\PiR{\sf \Pi}^\top]\GG \\
&=&\GG-\GG[{\sf \Pi}{\sf \Pi}^\top-{\sf \Pi}\PiR^\top{\sf M_R}\PiR{\sf \Pi}^\top]\GG,
\end{eqnarray*}
where ${\sf M_R}=(\GR^{-1}+\PiR\PiR^\top)^{-1}$. \\[1ex]
 Observe that both ${\sf G_R}$ and $ {\sf \Pi_R}^{\intercal}{\sf \Pi_R}$ are, respectively,  the circulant matrices
$${\sf G_R}=-\frac{d}{2cm}\mbox{\rm circ}(0,m-1,\ldots,j(m-j),\ldots,m-1),$$
$$ {\sf \Pi_R}^{\intercal}{\sf \Pi_R} =\dfrac{a}{m}\Big(m{\sf I}_m-{\sf J}_m\Big)=\dfrac{a}{m}\mbox{\rm circ}\left(m-1,-1,\ldots,-1\right).$$
Thus, for computing ${\sf G_R}^{-1}$ and ${\sf G_R}^{-1}+ {\sf \Pi_R}^{\intercal}{\sf \Pi_R}$, we use Lemma \ref{lemma.GR-1}.
\item Now, from \cite[Theorem 3.5]{CEGJM15}, we obtain $\sf{M_R}$  in Lemma \ref{lemma.MR}.
\item We next consider    ${\sf F}={\sf \Pi}{\sf \Pi}^\top-{\sf \Pi}\PiR^\top{\sf M_R}\PiR{\sf \Pi}^\top$.\ Notice that ${\sf F}$ is a matrix which non-zero elements are the circulant submatrix  $$\frac{a}{m}\mbox{\rm circ}\left(m-1,-1,\dots,-1\right).$$ It turns out that, ${\sf F}$ is the block Toeplitz matrix described in Lemma \ref{lemma.F}.
\item Since, $\MM=\GG-\GG \FF \GG$, we first compute  matrix $\HH=\GG\FF$, see Lemma \ref{lemma.H}.
\item Next, we compute matrix ${\sf K}=\GG \FF \GG=\HH\GG$ in Lemma \ref{lemma.K}. 
\item  Again, $\MM$ is a block Toeplitz matrix whose expression is obtained in Proposition \ref{Prop.M}.
\item Finally, using Lemma \ref{LemmaSG}, the claimed result of Theorem \ref{theo.Main} follows.
\end{enumerate}

\end{proof}
%

\section{Effective resistances and Kirchhoff index for non--complete wheels}

The Green matrix is a fundamental tool for computing some desired parameters of the network, like the effective resistances of the network or the Kirchhoff index, very useful in electric circuit theory or in organic chemistry, as natural indexes describing important structural properties of circuits or molecules.
Therefore, once we compute the Green matrix, the effective resistances between any two vertices of the new network are easily obtained from Formula \eqref{resistance} obtained in \cite{BCEG09}.

\begin{prop}\label{resistances}
Given two vertices of $\Gamma'$, the effective resistances of the network between them are:
\begin{itemize}
\item[a)] if $i,j\neq n+1$, $i=(k_1-1)d+h_1$, $j=(k_2-1)d+h_2$ and $k'=k_2-k_1+1$, where $1\leq k_1\leq k_2\leq n$ and $h_1,h_2=1,\dots,d$, with $h_1<h_2$ when $k_1=k_2$, then
$$\begin{array}{rl}
\hspace{-.25cm} R(i,j)=
&\hspace{-.25cm}\displaystyle 
\frac{1}{2cn}\Bigg(\dfrac{-1}{T_m(q)-1}\Bigg[ 
2n(h_2-h_1) \left( V_{k'-1}(q)  -V_{m-k'}(q) \right)
\\[3ex]
\hspace{-.25cm}\displaystyle- &\hspace{-.25cm}\displaystyle
  2\Big(\frac{an}{c} (h_1-1) (h_2-1-nd) -nd\Big)  \left( U_{k'-2}(q) +U_{m-k'}(q) \right)
 \\[3ex]
 \hspace{-.25cm}\displaystyle+&\hspace{-.25cm}\displaystyle
\Big(\frac{an}{c}((h_1-1)^2+(h_2-1)^2-d(h_1+h_2-2))-2dn\Big)U_{m-1}(q)\Bigg]
  \\[3ex] 
  +&\hspace{-.25cm}\displaystyle 2\left| (k'-1)d+h_2-h_1\right|\left(n-\left| (k'-1)d+h_2-h_1\right|\right)
  \\[3ex]
 \hspace{-.25cm}\displaystyle -&\hspace{-.25cm}\displaystyle
 2d(k'-1)\left[n-d(k'-1)+4d(h_1-h_2)\right]
 \\[3ex]
 \hspace{-.25cm}\displaystyle +&\hspace{-.25cm}\displaystyle
 2(h_1-h_2)[h_1-h_2+n]\Bigg),
  \end{array}$$
\item[b)] if $i=(k_1-1)d+h_1\neq n+1$ and $j=n+1$, then
$$R(i,n+1)=-\frac{1}{2c(T_m(q)-1)}\left(\frac{a}{c}(h_1-1)(h_1-1-d)-d\right)U_{m-1}(q).$$

\end{itemize}
with $q=\dfrac{ad}{2c}+1$.
Moreover, the Kirchhoff index of $\Gamma'$ is
$$
K(\Gamma')
=\dfrac{n(n+1)(6cd+a(d^2-1))}{12c^2}\dfrac{U_{m-1}(q)}{T_m(q)-1}
-\frac{(d^2-1)an+12cdn}{12ac}.
$$
\end{prop}

\begin{proof}
Firstly we point out that for any $i=1,\dots,n$ we can consider that $i=(k_1-1)d+h_1$, for any $1\leq k_1\leq m$, $1\leq h_1\leq d$, and thus any element of the main diagonal $ (({\sf L}')^{\#})_{ii}=({\sf L}'_{11})_{ii}=({\sf N}_{1})_{h_1h_1}$. Besides we know that $( ({\sf L}')^{\#})_{n+1\,n+1}={\sf L}'_{22}$.
Therefore, for any $i,j\neq n+1$, without loss of generality we can assume that $i=(k_1-1)d+h_1$, $j=(k_2-1)d+h_2$ for any $1\leq k_1\leq k_2\leq m$ and $h_1,h_2=1,\dots,d$, with $h_1<h_2$ when $k_1=k_2$, and if $k'=k_2-k_1+1$. Then
\begin{eqnarray*}
 R(i,j)&=& ({\sf L}'_{11})_{ii}+({\sf L}'_{11})_{jj}-2({\sf L}'_{11})_{ij}\\
        &=&({\sf N}_{1})_{h_1h_1}+({\sf N}_{1})_{h_2h_2}-2({\sf N}_{k'})_{h_1h_2},
\end{eqnarray*}
and for $i=(k_1-1)d+h_1\neq n+1$ and $j=n+1$,
\begin{eqnarray*}
 R(i,n+1)&=& ({\sf L}'_{11})_{ii}+{\sf L}'_{22}-2({\sf L}'_{12})_{i}\\
        &=&({\sf N}_{1})_{h_1h_1}+{\sf L}'_{22}-2({\sf L}'_{12})_{h_1}.
\end{eqnarray*}
By simplifying the previous expressions we obtain the claimed result, and besides the Kirchhoff index of the non--complete wheel $(n,m)$-$W$ is obtained by simplifying the expression
$$\begin{array}{rl}
K(\Gamma')=&\hspace{-.25cm}\displaystyle (n+1){\rm trace}({\sf L}')=(n+1){\rm trace}({\sf L}'_{11})+(n+1) {\sf L}'_{22}.
 \end{array}$$
\end{proof}

Finally, if we consider the case where $d=1$, $n=m$, we obtain a complete Wheel, and in this case we notice that the Kirchhoff index coincides with the result obtained in \cite{BCEG08}.
\begin{corollary}
The Kirchhoff index of a Wheel on $n+1$ vertices is
$$
K(\Gamma')
=\dfrac{n(n+1)}{2c}\dfrac{ U_{n-1}(q)}{T_n(q)-1}-\dfrac{n}{a}.
$$

\end{corollary}

\section{Technical lemmas}
In this section we include the technical lemmas and their proofs. Observe that the first Lemma excludes the case of adding a pendant vertex ($m=1$). Besides, we point out that the result of Lemma~\ref{lemma.MR} is obtained from the application of Theorem 3.5. from \cite{CEGJM15}.
\begin{lemma}\label{lemma.GR-1}
The inverse matrix of ${\sf G_R}$ is the circulant matrix
$${\sf G_R}^{-1}=\dfrac{12c}{n(m^2-1)}\mbox{\rm circ}\left(b_0,-b_1,-1,\ldots,-1,-b_1\right),$$
where $b_0=(m^3-m-6)/6$ and $b_1=(m^3-m+12)/12$.
Moreover,
$${\sf G_R}^{-1}+ {\sf \Pi_R}^{\intercal}{\sf \Pi_R}=\dfrac{1}{n(m^2-1)}\mbox{\rm circ}\left(c_0,-c_1,-c_2,\ldots,-c_2,-c_1\right),$$
where
$$c_0=12cb_0+ad(m-1)(m^2-1), c_1=12cb_1+ad(m^2-1), c_2=12c+ad(m^2-1).$$
\end{lemma}
\begin{proof}
Taking into account that $\sum_{k=1}^m k(m-k)=-(m^3-m)/3$, the first part of the result can be checked by simple multiplication of both matrices. The second part of the statement is straightforward.
\end{proof}
\begin{lemma}\label{lemma.MR}
The circulant matrix  ${\sf G_R}^{-1}+ {\sf \Pi_R}^{\intercal}{\sf \Pi_R}$ is invertible iff $m>1$
and in this case
$${  \sf M_R}=({\sf G_R}^{-1}+ {\sf \Pi_R}^{\intercal}{\sf \Pi_R})^{-1}=\mbox{\rm circ}(t_1,\dots,t_m),$$
where
$$t_j=\left[\frac{[U_{j-2}(q)+U_{m-j}(q)]d}{2c[T_m(q)-1]}-\frac{12c+ad(m^2-1)}{12acm}\right], \quad j=1,\dots,m,$$
with $q=\dfrac{ad}{2c}+1$.
\end{lemma}

\begin{lemma}\label{lemma.F}
The matrix  ${\sf F}$ is a block Toeplitz matrix given by
$$
 \FF=
 \begin {bmatrix}\FF_1 & \FF_2 & \dots & \FF_m\\[1ex]
 \FF_m & \FF_{1} & \ddots& \vdots\\[1ex]
\vdots & \ddots &\ddots&  \FF_2 \\[1ex]
 \FF_2 & \dots & \FF_m& \FF_1
 \end {bmatrix}
 $$
where each submatrix ${\sf F}_i$ has all its elements equal $0$ except the first one, and
$$({\sf F}_i)_{11}=f_i=a\delta_{i1}-\frac{a^2d}{2c}\frac{\left[U_{i-2}(q)+U_{m-i}(q)\right]}{\left[T_m(q)-1\right]},  \; for \; any\; i=1,\dots,m,$$
with $q=\dfrac{ad}{2c}+1$.
\end{lemma}

\begin{proof}
Observe that
\begin{eqnarray*}
f_1&=&  \frac{a(m-1)}{m}+\frac{a^2}m\sum_{j=1}^m t_j-a^2 t_1,\\
f_i&=&  -\frac{a}{m}+\frac{a^2}m\sum_{j=1}^m t_j-a^2 t_i,  \;\;  for \; any \; i=2,\dots,d.\\
\end{eqnarray*}
Therefore, we first compute
\begin{eqnarray*}
\sum_{j=1}^m t_j&=&\sum_{j=1}^m \left[\frac{[U_{j-2}(q)+U_{m-j}(q)]d}{2c[T_m(q)-1]}-\frac{12c+ad(m^2-1)}{12acm}\right]\\
&=&\frac{d}{2c[T_m(q)-1]}\left[\sum_{j=1}^m U_{j-2}(q)+U_{m-j}(q)\right]-\frac{12c+ad(m^2-1)}{12ac}\\
&=&\frac{d}{2c[T_m(q)-1]}\frac{2c}{ad}[T_m(q)-1]-\frac{12c+ad(m^2-1)}{12ac}=\frac{-d}{12c}(m^2-1).\\
\end{eqnarray*}
And now we have
\begin{eqnarray*}
f_1&=&  \frac{a(m-1)}{m}-\frac{a^2d(m^2-1)}{12cm}-a^2 \left[\frac{[U_{m-1}(q)]d}{2c[T_m(q)-1]}-\frac{12c+ad(m^2-1)}{12acm}\right]\\
&=& a-\frac{a^2d}{2c} \frac{U_{m-1}(q)}{\left[T_m(q)-1\right]},\\
\end{eqnarray*}
and for any $i=2,\dots,d$,
\begin{eqnarray*}
f_i&&=-\frac{a}{m}-\frac{a^2d(m^2-1)}{12cm}-a^2 \left[\frac{[U_{i-2}(q)+U_{m-i}(q)]d}{2c[T_m(q)-1]}-\frac{12c+ad(m^2-1)}{12acm}\right]\\
&&=- \frac{a^2d}{2c}\frac{\left[U_{i-2}(q)+U_{m-i}(q)\right]}{\left[T_m(q)-1\right]}.
\end{eqnarray*}
\end{proof}
\begin{lemma}\label{lemma.H}
The matrix  ${\sf H}$ is a block Toeplitz matrix given by
$$
 \HH=
 \begin {bmatrix}{\sf H}_1 & {\sf H}_2  & \dots & {\sf H}_m \\[1ex]
 {\sf H}_m  & {\sf H}_1 & \ddots& \vdots\\[1ex]
\vdots & \ddots &\ddots&  {\sf H}_2  \\[1ex]
{\sf H}_2  & \dots & {\sf H}_m  & {\sf H}_1
 \end {bmatrix}
 $$
where each submatrix ${\sf H}_k$ has all its elements equal $0$ except its first column, and for any $i=1,\dots,d$
$$\begin{array}{rl}
\hspace{-.25cm}({\sf H}_1)_{i1}=&\hspace{-.25cm}\displaystyle-\dfrac{d}{n}-\dfrac{a}{2c}\dfrac{1}{\left[T_m(q)-1\right]}
\Big[(i-1)(V_{m-1}(q)-1)-dU_{m-1}(q)\Big]
,\\[3ex]
\hspace{-.25cm}({\sf H}_k)_{i1}=&\hspace{-.25cm}\displaystyle-\dfrac{d}{n}-\dfrac{a}{2c}\dfrac{1}{\left[T_m(q)-1\right]}\Big[(i-1)(V_{k-1}(q)-V_{m-k}(q))\\[3ex]
-&\hspace{-.25cm}\dfrac{d}{c}[a(i-1)+c](U_{k-2}(q)+U_{m-k}(q))\Big],  \; for \;k=2,\dots,d,
\end{array}$$
with $q=\dfrac{ad}{2c}+1$.
\end{lemma}
\begin{proof}
We point out that $\GG$ is also a block Toeplitz matrix and thus $\HH$ is again a block Toeplitz matrix, and hence we need to compute just the first $d$ rows. We first define the following summations:
$$ s_1(i,j)= \sum\limits_{p=i}^{j}f_{p},\quad  s_2(i,j)= \sum\limits_{p=i}^{j}pf_{p}\quad \textrm{and}\quad
s_3(i,j)= \sum\limits_{p=i}^{j}p^2f_{p}.
$$
In particular,
$$\begin{array}{rl}
 \displaystyle K_1=s_1(1,m)=&\hspace{-.25cm}\displaystyle a-\frac{a}{2}\frac{1}{\left[T_m(q)-1\right]}2\left[T_m(q)-1\right]=0,\\[3ex]
 \displaystyle K_2=s_2(1,m)=&\hspace{-.25cm}\displaystyle -\frac{a}{2}\frac{m}{\left[T_m(q)-1\right]}\Big[ V_{m-1}(q)- 1 \Big],\\[3ex]
 \displaystyle K_3=s_3(1,m)=&\hspace{-.25cm} \displaystyle = -\frac{2c}{d}-\frac{a}{2}\frac{1}{\left[T_m(q)-1\right]}\Big[ m^2 V_{m-1}(q)
  -  2m U_{m-2}(q) - m(m+2) \Big].\\
 \end{array}
$$
And moreover,
$$\begin{array}{rl}
\hspace{-.25cm}B_1&\displaystyle  =s_1(k+1,m)=-\frac{a}{2}\frac{1}{\left[T_m(q)-1\right]}\Big[ V_{m-1}(q)- V_{k-1}(q)  +V_{m-k}(q)-1\Big], \\[3ex]
\hspace{-.25cm}B_2&\displaystyle =s_2(k+1,m)=-\frac{a}{2}\frac{1}{\left[T_m(q)-1\right]}\Big[m V_{m-1}(q)+k(V_{m-k}(q)-V_{k-1}(q)) \\[3ex]
 \hspace{-.25cm}&\displaystyle +U_{k-2}(q)+U_{m-k}(q) -U_{m-2}(q)-(m+1)\Big].\\
 \end{array}
$$
Now observe that when $j=1+(k-1)d$ and $k=1,\ldots,m$,  for $i=1,\ldots,d$ we have that
$$\begin{array}{rl}
({\sf H}_k)_{i1}=&\hspace{-.25cm}\displaystyle\sum\limits_{\ell=1}^{k}g_{i,1+(\ell-1)d}f_{k-\ell+1}+\sum\limits_{\ell=k+1}^{m}g_{i,1+(\ell-1)d}f_{m+k-\ell+1}\\[3ex]
=&\hspace{-.25cm}\displaystyle g_{i,1}f_{k}+\sum\limits_{\ell=2}^{k}g_{i,1+(\ell-1)d}f_{k-\ell+1}+\sum\limits_{\ell=k+1}^{m}g_{i,1+(\ell-1)d}f_{m+k-\ell+1}
\\[3ex]
=&\hspace{-.25cm}\displaystyle A_1+ A_2+ A_3.
\\[3ex]
\end{array}
$$
Firstly we suppose $k>1$. We compute separately $A_2$ and $A_3$ as follows
$$\begin{array}{rl}
\hspace{-.25cm}\displaystyle A_{21}=&\sum\limits_{\ell=2}^{k}f_{k-\ell+1}= \sum\limits_{p=1}^{k-1}f_{p}=s_1(1,k-1),\\[3ex]
\hspace{-.25cm}\displaystyle A_{22}=&\sum\limits_{p=1}^{k-1}(k-p)f_{p}=j s_1(1,k-1) -  s_2(1,k-1),\\[3ex]
\hspace{-.25cm}\displaystyle A_{23}= &\sum\limits_{p=1}^{k-1}(k-p)^2f_{p}=
k^2 s_1(1,k-1) - 2k s_2(1,k-1)+s_3(1,k-1),
\\[3ex]
\end{array}
$$
and thus,
$$\begin{array}{rl}
A_2=\hspace{-.25cm}&\displaystyle
 \sum\limits_{\ell=2}^{k}g_{i,1+(\ell-1)d}f_{k-\ell+1}= \dfrac{1}{12cn} \Big[\Big(n^2-1+6(i-1)(n+i-1)\Big)A_{21}\\[3ex]
&\hspace{-.25cm}\displaystyle-6d(n+2i-2)A_{22}+6d^2A_{23}\Big]=\dfrac{1}{12cn}\Big[\alpha_1A_{21}+\alpha_2 A_{22}+\alpha_3 A_{23}\Big]\\[3ex]
=&\hspace{-.25cm}\displaystyle  \dfrac{1}{12cn}\Big[( \alpha_1+\alpha_2k+\alpha_3 k^2)s_1(1,k-1)-(\alpha_2+ 2\alpha_3 k)s_2(1,k-1)\\[3ex]
+&\hspace{-.25cm}\displaystyle \alpha_3s_3(1,k-1)\Big].\\[3ex]
\end{array}
$$
Besides,
$$\begin{array}{rl}
\hspace{-.25cm}\displaystyle A_{31}=&\sum\limits_{\ell=k+1}^{m}f_{m+k-\ell+1}=\sum\limits_{p=k+1}^{m}f_{p}=s_1(k+1,m),\\[3ex]
\hspace{-.25cm}\displaystyle A_{32}=&\sum\limits_{\ell=k+1}^{m}(\ell-1)f_{m+k-\ell+1}=(m+k)s_1(k+1,m)-s_2(k+1,m),\\[3ex]
\hspace{-.25cm}\displaystyle A_{33}=&\sum\limits_{\ell=k+1}^{m}(\ell-1)^2f_{m+k-\ell+1}=
\\[3ex]
=&\hspace{-.25cm}\displaystyle(m+k)^2s_1(k+1,m)-2(m+k)s_2(k+1,m)+s_3(k+1,m),
\end{array}
$$
and thus,
$$\begin{array}{rl}
A_3=\hspace{-.25cm}&\hspace{-.25cm}\displaystyle
 \sum\limits_{\ell=k+1}^{m}g_{i,1+(\ell-1)d}f_{m+k-\ell+1}= \dfrac{1}{12cn}\Big[\alpha_1A_{31}+\alpha_2 A_{32}+\alpha_3 A_{33}\Big]\\[3ex]
= \dfrac{1}{12cn}&\hspace{-.25cm}\displaystyle \Big[( \alpha_1+\alpha_2 (m+k)+\alpha_3 (m+k)^2)s_1(k+1,m)\\[3ex]
&\hspace{-.25cm}\displaystyle -(\alpha_2+2\alpha_3 (m+k))s_2(k+1,m)+ \alpha_3s_3(k+1,m)\Big]\\[3ex]
= \dfrac{1}{12cn}&\hspace{-.25cm}\displaystyle \Big[( \alpha_1+\alpha_2k+\alpha_3 k^2)s_1(k+1,m) -(\alpha_2+2\alpha_3 k)s_2(k+1,m)\\[3ex]
&\hspace{-.25cm}\displaystyle+ \alpha_3s_3(k+1,m)+\left(\alpha_2m+\alpha_3m^2+2\alpha_3mk\right)B_1-2m\alpha_3B_2\Big].\\[3ex]
\end{array}
$$
Now adding all $A_1+A_2+A_3=$
$$\begin{array}{rl}
= \dfrac{1}{12cn}&\hspace{-.25cm}\displaystyle \Big[\left(n^2-1-6(i-1)(n-i+1)\right)f_{k}+( \alpha_1+\alpha_2k+\alpha_3 k^2)(-f_{k}) \\[3ex]
-&\hspace{-.25cm}\displaystyle
(\alpha_2+2\alpha_3 k)\left(K_2-kf_k\right)+ \alpha_3\left(K_3-k^2f_k\right)-2m\alpha_3B_2
\\[3ex]
+&\hspace{-.25cm}\displaystyle \left(\alpha_2m+\alpha_3m^2+2\alpha_3mk\right)B_1\Big]
\\[3ex]
= \dfrac{1}{12cn}&\hspace{-.25cm}\displaystyle \Big[(n^2-1-6(i-1)(n-i+1)- \alpha_1)f_{k}-(\alpha_2+2\alpha_3 k)K_2+ \alpha_3 K_3 \\[3ex]
&\hspace{-.25cm}\displaystyle
-2m\alpha_3B_2+ \left(\alpha_2m+\alpha_3m^2+2\alpha_3mk\right)B_1\Big]
\\[3ex]
=- \dfrac{d}{n}&\hspace{-.25cm}\displaystyle - \dfrac{a}{2c}\frac{1}{\left[T_m(q)-1\right]}\Big[(i-1)(V_{k-1}(q)-V_{m-k}(q))
\\[3ex]
&\hspace{-.25cm}\displaystyle \hspace{4cm}
-\frac{d}{c}\left(a(i-1)+c\right)\left(U_{k-2}(q)+U_{m-k}(q)\right)\Big].
\\[3ex]
\end{array}
$$
Now observe that if $k=1$ then $j=1$, and in this case $A_2=0$. Besides it holds $s_1(2,m)=-f_1$, $s_2(2,m)=K_2-f_1$ and
$s_3(2,m)=K_3-f_1$.

Thus
$$\begin{array}{rl}
({\sf H}_1)_{i1}
=&\hspace{-.25cm}\displaystyle g_{i,1}f_{1}+\sum\limits_{\ell=2}^{m}g_{i,1+(\ell-1)d}f_{m+2-\ell}
\\[3ex]
= \dfrac{1}{12cn}&\hspace{-.25cm}\displaystyle \Big[\left(n^2-1-6(i-1)(n-i+1)\right)f_{1}- ( \alpha_1+\alpha_2 (m+1)+\\[3ex]
&\hspace{-.25cm}\displaystyle \alpha_3 (m+1)^2)f_1-(\alpha_2+2\alpha_3 (m+1))(K_2-f_1)+ \alpha_3(K_3-f_1)\Big]\\[3ex]
= \dfrac{1}{12cn}&\hspace{-.25cm}\displaystyle \Big[(n^2-1-6(i-1)(n-i+1)- \alpha_1-\alpha_2 (m+1)- \alpha_3 (m+1)^2\\[3ex]
&\hspace{-.25cm}\displaystyle+\alpha_2+2\alpha_3 (m+1) -\alpha_3)f_1-(\alpha_2+2\alpha_3 (m+1))K_2+ \alpha_3K_3\Big]\\[3ex]
= \dfrac{1}{12cn}&\hspace{-.25cm}\displaystyle \Big[(n^2-1-6(i-1)(n-i+1)- \alpha_1-\alpha_2 m- \alpha_3m^2)f_1\\[3ex]
&\hspace{-.25cm}\displaystyle-(\alpha_2+2\alpha_3 (m+1))K_2+ \alpha_3K_3\Big]\\[3ex]
= \dfrac{1}{12cn}&\hspace{-.25cm}\displaystyle \Big[(12d(i-1)-6dn-12d^2)K_2+6d^2K_3\Big]\\[3ex]
=-\dfrac{d}{n} &\hspace{-.25cm}\displaystyle-\dfrac{a}{2c}\dfrac{1}{\left[T_m(q)-1\right]}
\Big[(i-1)(V_{m-1}(q)-1)-dU_{m-1}(q)\Big].\\
\end{array}
$$
\end{proof}

\begin{lemma}\label{lemma.K}
The matrix  ${\sf K}$ is a block Toeplitz matrix given by
$$
{\sf K}=
 \begin {bmatrix}{\sf K}_1& {\sf K}_2  & \dots & {\sf K}_m \\[1ex]
 {\sf K}_m  & {\sf K}_1 & \ddots& \vdots\\[1ex]
\vdots & \ddots &\ddots&  {\sf K}_2  \\[1ex]
{\sf K}_2  & \dots & {\sf K}_m  & {\sf K}_1
 \end {bmatrix}
 $$
where each submatrix ${\sf K}_k$ has as elements
$$\begin{array}{rl}
\hspace{-.25cm}({\sf K}_k)_{i,h+1}=
&\hspace{-.25cm}\displaystyle -\dfrac{1}{2cn}\dfrac{1}{T_m(q)-1}
\bigg[ n \left( h-i+1 \right)  \left( V_{k-1}(q)  -V_{m-k}(q) \right)
\\[3ex]
\hspace{-.25cm}\displaystyle- &\hspace{-.25cm}\displaystyle
  \Big(\frac{an}{c} ( i-1) ( h-d) -nd \Big)  \left( U_{k-2}(q) +U_{m-k}(q) \right)
 \\[3ex]
 \hspace{-.25cm}\displaystyle- &\hspace{-.25cm}\displaystyle
 \frac{2cd}{a} \left( V_{m-1}(q) -1\right) - d^2 U_{m-1}(q) \bigg] +dh+nh-nd-\frac{n^2}{6}+\frac{d^2}{6}
 \\[3ex]
 \hspace{-.25cm}\displaystyle +&\hspace{-.25cm}\displaystyle
 \left( 2\,kd-3\,d+2\,h-n \right)  \left( i-1\right) +kd \left( n-2\,h \right)-d^2(k-1)^2,
 \\
 \end{array}$$
with $q=\dfrac{ad}{2c}+1$, for any $k=1,\dots,m$, $i=1,\dots,d$ and $h=0,\dots,d-1$.
\end{lemma}
\begin{proof}
Again we multiply two block Toeplitz matrices and thus ${\sf K}$ is a block Toeplitz matrix, and hence we can just compute only the first $d$ rows. For each submatrix ${\sf K}_k$, $k=1,\dots,m$, we compute its elements. We point out that the $i(h+1)$-element of  ${\sf K}_k$ is the $ij$-element of ${\sf K}$, with $j=1+(k-1)d+h$, for any   $h=0,\dots,d-1$.
$$\begin{array}{rl}
({\sf K}_k)_{i,h+1}=&\hspace{-.25cm}\displaystyle\sum\limits_{\ell=1}^{m}h_{i,1+(\ell-1)d}\cdot g_{1+(\ell-1)d,j}= g_{i,1}h_{i,1}+ \sum\limits_{\ell=2}^{k}h_{i,1+(\ell-1)d}\cdot g_{1+(\ell-1)d,j}\\[3ex]
+&\hspace{-.25cm}\displaystyle\sum\limits_{\ell=k+1}^{m}h_{i,1+(\ell-1)d}\cdot g_{1+(\ell-1)d,j}= A_1+ A_2+ A_3.
\\[3ex]
\end{array}
$$
We compute separately
$$\begin{array}{rl}
\hspace{-.25cm}\displaystyle A_1=\frac{1}{12cn}&\hspace{-.25cm}\displaystyle h_{i,1}\left[n^2-1-6(h+kd-d)(n-h-kd+d)\right],\\[3ex]
\hspace{-.25cm}\displaystyle A_2=\frac{1}{12cn}&\hspace{-.25cm}\displaystyle \sum_{\ell=2}^k  h_{i,1+(\ell-1)d}\left[n^2-1-6(h+kd-\ell d)(n-h-kd+\ell d)\right]\\[3ex]
\hspace{-.25cm}\displaystyle =\frac{1}{12cn}&\hspace{-.25cm}\displaystyle \sum_{\ell=2}^k  h_{i,1+(\ell-1)d}\big[n^2-1-6(-h^2-2kdh-k^2d^2+2(kd^2+dh)\ell\\[3ex]
\hspace{-.25cm}\displaystyle &\hspace{-.25cm}\displaystyle -d^2\ell^2 + n(h+kd) -\ell dn)\big],
\end{array}
$$
$$\begin{array}{rl}
\hspace{-.25cm}\displaystyle A_3=\frac{1}{12cn}&\hspace{-.25cm}\displaystyle  \sum_{\ell=k+1}^m  h_{i,1+(\ell-1)d}
\left[n^2-1-6(-h-kd+\ell d)(n+h+kd-\ell d)\right],\\[3ex]
\hspace{-.25cm}\displaystyle =\frac{1}{12cn}&\hspace{-.25cm}\displaystyle  \sum_{\ell=k+1}^m  h_{i,1+(\ell-1)d}\big[n^2-1-6(-h^2-2kdh-k^2d^2+2(kd^2+dh)\ell\\[3ex]
\hspace{-.25cm}\displaystyle &\hspace{-.25cm}\displaystyle -d^2\ell^2-n(h+kd) + \ell dn)\big].
\end{array}
$$
Besides,
$$\begin{array}{rl}
\hspace{-.25cm}\displaystyle B_{1}=&\hspace{-.25cm}\displaystyle\sum\limits_{\ell=2}^{k}h_{i,1+(\ell-1)d}= -\frac {d(k-1)}{n}-\frac{a}{2c}\frac{1}{( T_m(q) -1)} \bigg[(i-1) \Big( U_{k-2}(q) +U_{m-k}(q)\\[3ex]
\hspace{-.25cm}\displaystyle &\hspace{-.25cm}\displaystyle-U_{m-1}(q) \Big) -\frac {c}{a} \Big(V_{k-1}(q)-V_{m-k}(q)+V_{m-1}(q)) -1\Big)\bigg],\\[3ex]
\hspace{-.25cm}\displaystyle B_{2}=&\hspace{-.25cm}\displaystyle\sum\limits_{\ell=2}^{k} \ell h_{i,1+(\ell-1)d)}= -\frac {d(k-1)(k+2)}{n}-\frac{a}{2c}\frac{1}{( T_m(q) -1)} \bigg[(i-1) \Big( (k+1)U_{k-2}(q) \\[3ex]
\hspace{-.25cm}\displaystyle +&\hspace{-.25cm}\displaystyle(k+1)U_{m-k}(q)-2U_{m-2}(q)-1-\frac{1}{2(q-1)}(V_{k-1}(q)-V_{m-k}(q)+V_{m-1}(q))\Big)
\\[3ex]
\hspace{-.25cm}-\displaystyle &\hspace{-.25cm}\displaystyle \frac {c}{a}\Big( k V_{k-1}(q) -kV_{m-k}(q)+V_{m-1}(q) -U_{k-2}(q)
-U_{m-k}(q)+U_{m-1}(q) \Big) \bigg],
\end{array}
$$
and
$$\begin{array}{rl}
\hspace{-.25cm}\displaystyle K_{1}=&\hspace{-.25cm}\displaystyle\sum\limits_{\ell=1}^{m}h_{i,1+(\ell-1)d}= 0,\\[3ex]
\hspace{-.25cm}\displaystyle K_{2}=&\hspace{-.25cm}\displaystyle\sum\limits_{\ell=2}^{k} \ell h_{i,1+(\ell-1)d)}=-
\frac{1}{(T_m(q) -1)} \bigg[\bigg( (i-1)\big(a-\frac{1}{d}-\frac{a}{2c}(m+2)\big)\frac{m+2}{2}\bigg)
 \\[3ex]
\hspace{-.25cm}\displaystyle &\hspace{-.25cm}\displaystyle \Big(V_{m-1}(q)-1\Big)+\Big(
\frac{a}{2c}(i-1)(m-1)-q+1\Big) U_{m-1}(q)\bigg]-\frac{m+1}{2},
\\[3ex]
\hspace{-.25cm}\displaystyle K_{3}=&\hspace{-.25cm}\displaystyle\sum\limits_{\ell=2}^{k} \ell^2 h_{i,1+(\ell-1)d)}=
-\frac{1}{(T_m(q) -1)} \bigg[-\bigg((i-1) \Big( \frac{m+3}{d} +\frac{anm}{2cd}+\frac{an}{cd}\Big)
\\[3ex]
\hspace{-.25cm}\displaystyle+&\hspace{-.25cm}\displaystyle\frac{m^2}{2}+m+1+\frac{2c}{ad}\bigg)
\Big(V_{m-1}(q)-1\Big)+\Big(\frac{a}{2c}(i-1)(m^2+2m-3)
\\[3ex]
\hspace{-.25cm}\displaystyle +&\hspace{-.25cm}\displaystyle m-q\Big)U_{m-1}(q)\bigg]-\frac{(m+1)(2m+1)}{2}.
\end{array}
$$
Now adding $A_1+A_2+A_3=$
$$\begin{array}{rl}
\hspace{-.25cm}\displaystyle =\frac{1}{12cn}&\hspace{-.25cm}\displaystyle(n^2-1) \sum_{\ell=1}^m  h_{i,1+(\ell-1)d}
-\frac{1}{2cn}h_{i,1}\left[-6(h+kd-d)(n-h-kd+d)\right]
\\[3ex]
\hspace{-.25cm}\displaystyle-\frac{1}{2cn}&\hspace{-.25cm}\displaystyle \sum_{\ell=2}^m h_{i,1+(\ell-1)d}(-h^2-2kdh-k^2d^2
+2(kd^2+dh)\ell-d^2\ell^2\\[3ex]
\hspace{-.25cm}\displaystyle-\frac{1}{2cn}&\hspace{-.25cm}\displaystyle \sum_{\ell=2}^k h_{i,1+(\ell-1)d}  n(h+kd) -\ell dn
+\frac{1}{2cn} \sum_{\ell=k+1}^m h_{i,1+(\ell-1)d}  n(h+kd) -\ell dn\\[3ex]
\hspace{-.25cm}\displaystyle=&\hspace{-.25cm}\displaystyle-\frac{1}{2cn}\Big[ h_{i,1}(h+kd-d)(n-h-kd+d)+(-h^2-2kdh-k^2d^2)(-h_{i,1})
\\[3ex]
+&\hspace{-.25cm}\displaystyle
2(kd^2+dh)(K_2-h_{i,1})-d^2(K_3-h_{i,1})+ n(h+kd) (2B_1+h_{i,1})
\\[3ex]
-&\hspace{-.25cm}\displaystyle
 dn(2B_2-K_2+h_{i,1})\Big]=-\frac{1}{2cn}\Big[2nh_{i,1}(kd+h-d)+ 2n(h+kd) B_1
\\[3ex]
\hspace{-.25cm}\displaystyle -&\hspace{-.25cm}\displaystyle 2nd B_2+(dn+2dh+2kd^2)K_2 -d^2 K_3\Big]
\\[3ex]
&\hspace{-.25cm}\displaystyle -\dfrac{1}{2cn}\dfrac{1}{T_m(q)-1}
\bigg[ n \left( h-i+1 \right)  \left( V_{k-1}(q)  -V_{m-k}(q) \right)
\\[3ex]
\hspace{-.25cm}\displaystyle- &\hspace{-.25cm}\displaystyle
  \Big(\frac{an}{c} ( i-1) ( h-d) -nd \Big)  \left( U_{k-2}(q) +U_{m-k}(q) \right)
 \\[3ex]
 \hspace{-.25cm}\displaystyle- &\hspace{-.25cm}\displaystyle
 \frac{2cd}{a} \left( V_{m-1}(q) -1\right) - d^2 U_{m-1}(q) \bigg] +dh+nh-nd-\frac{n^2}{6}+\frac{d^2}{6}
 \\[3ex]
 \hspace{-.25cm}\displaystyle +&\hspace{-.25cm}\displaystyle
 \left( 2\,kd-3\,d+2\,h-n \right)  \left( i-1\right) +kd \left( n-2\,h \right)-d^2(k-1)^2.
\end{array}$$
Finally, taking into account that $qU_{m-1}(q)-U_{m-2}(q)-1=T_{m}(q)-1$, we obtain the desired result.
\end{proof}
\begin{prop}\label{Prop.M}
The group inverse of the Schur complement of $\LL$ is a block Toeplitz matrix given by
$$
{\sf M}=
 \begin {bmatrix}{\sf M}_1 & {\sf M}_2  & \dots & {\sf M}_m\\[1ex]
 {\sf M}_m  & {\sf M}_1 & \ddots& \vdots\\[1ex]
\vdots & \ddots &\ddots&  {\sf M}_2  \\[1ex]
{\sf M}_2  & \dots & {\sf M}_m & {\sf M}_1
 \end {bmatrix}
 $$
where each submatrix ${\sf M}_k$ has as elements
$$\begin{array}{rl}
\hspace{-.25cm}({\sf M}_k)_{i,h+1}=
&\hspace{-.25cm}\displaystyle \dfrac{1}{2cn}\Bigg(-\left| i-1-d(k-1)-h\right|\cdot(n-\left| i-1-d(k-1)-h\right|)
\\[3ex]
+&\hspace{-.25cm}\displaystyle\dfrac{1}{T_m(q)-1}
\bigg[ n \left( h-i+1 \right)  \left( V_{k-1}(q)  -V_{m-k}(q) \right)
\\[3ex]
\hspace{-.25cm}\displaystyle- &\hspace{-.25cm}\displaystyle
  \Big(\frac{an}{c} ( i-1) ( h-d) -nd \Big)  \left( U_{k-2}(q) +U_{m-k}(q) \right)
 \\[3ex]
 \hspace{-.25cm}\displaystyle- &\hspace{-.25cm}\displaystyle
 \frac{2cd}{a} \left( V_{m-1}(q) -1\right) - d^2 U_{m-1}(q) \bigg] +dh+nh-nd
 \\[3ex]
 \hspace{-.25cm}\displaystyle +&\hspace{-.25cm}\displaystyle
 \left( 2\,kd-3\,d+2\,h-n \right)  \left( i-1\right) +kd \left( n-2\,h \right)-d^2(k-1)^2\Bigg)
 \\[3ex]
 \hspace{-.25cm}\displaystyle +&\hspace{-.25cm}\displaystyle
 \dfrac{d^2-1}{12cn},
\end{array}$$
with $q=\dfrac{ad}{2c}+1$, for any $k=1,\dots,m$, $i=1,\dots,d$ and $h=0,\dots,d-1$.
\end{prop}
\begin{lemma} \label{LemmaSG}
The product $\GG\s={\sf n} \otimes {\sf j}_m$ is a column vector where
 $$
({\sf n})_{i}=-\frac{a}{12cn}\Big(n(6+d)+5m+6mi(i-d-2)\Big), \; \textrm{for any}\;\; i=1,\dots,d.
 $$
Besides, it holds $\MM\s=\GG\s.$
\end{lemma}
\begin{proof} For a given row $i=1,\dots,d$, we have
$$\begin{array}{rl}
\hspace{-.25cm}(\GG\s)_i=
&\hspace{-.25cm}\displaystyle
-a\sum_{k=1}^m g_{i,1+(k-1)d}=-ag_{i,1}-a\sum_{k=2}^m g_{i,1+(k-1)d}=-\frac{a}{12cn}(n^2-1)m
\\[3ex]
-&\hspace{-.25cm}\displaystyle
\frac{a}{2cn}\bigg((i-1)(n-i+1)+\sum_{k=2}^m (1+(k-1)d-i)(n-1-(k-1)d+i)\bigg)
\\[3ex]
\hspace{-.25cm}\displaystyle
=&\hspace{-.25cm}\displaystyle
\frac{-a}{12cn}\Big(6n+5m+dn-6mi(d+2-i) \Big).
 \\[3ex]
\end{array}$$
Moreover, as $\GG$ is a block Toeplitz matrix, it holds $\d (\GG\s)_i= (\GG\s)_{i+(k-1)d}$, for any $k=1,\dots,m$.

Besides,
$\MM\s=(\GG-\GG \FF \GG)\s=\GG\s-\GG \FF \GG\s=\GG\s$, as it is
straightforward to verify that $ \FF \GG\s=(0,\dots,0)^\top$.
\end{proof}

%






\end{document}